\documentclass[12pt]{amsart}

\usepackage{tikz}
\usepackage{wrapfig}
\usepackage{caption} 
\usepackage{enumitem}
\usepackage{mathtools,stackengine} 
\usetikzlibrary{decorations.text}
\usetikzlibrary{positioning}
\usepackage{graphicx,lipsum,pgfplots}
\usepackage{float}
\usepackage{hyperref} 
\usepackage{pifont}
\usepackage{geometry,a4wide,amssymb,amsfonts,amsrefs,amsmath}    
\usepackage{amsaddr}
\usepackage{bbm,xcolor}
\usepackage{dsfont}
\usepackage{yfonts}
\usepackage[T1]{fontenc}
\usepackage{mathtools}
\usepackage{mathrsfs}

\newcommand\mystyle{\everymath{\displaystyle}}
\mystyle
\newcommand{\Z}{\mathbb{Z}}

\newcommand{\N}{\mathbb{N}}
\newcommand{\R}{\mathbb{R}}

\newcommand{\M}{\mathcal{M}}

\newcommand{\G}{\mathscr{G}}

\newcommand{\diam}{\text{\textsf{\text{diam}}}}

\newcommand{\es}{\mathcal{ES}}

\newcommand{\la}{\mathtt{a}}
\newcommand{\bla}{\mathtt{\mathbf{a}}}
\newcommand{\bL}{\mathbb{L}}

\def\mybf #1{{\bf{\emph{#1}}}}

\stackMath
\newcommand\frightarrow{\scalebox{1}[.3]{$\rule[.45ex]{3ex}{1.5pt}%
  \kern-.2ex{\blacktriangleright}$}}
\newcommand\darrow[1][]{\mathrel{\stackon[-22pt]{\stackanchor[4pt]{\frightarrow}{\frightarrow}}{\scriptstyle#1}}}

\def\mybf #1{\textbf{\textit{#1}}}

\definecolor{forestgreen(web)}{rgb}{0.13, 0.55, 0.13}

\theoremstyle{plain}

\newtheorem{thm}{Theorem}[section]
\newtheorem{lemma}[thm]{Lemma}
\newtheorem{prop}[thm]{Proposition}

\newtheorem{defin}[thm]{Definition}

\newtheorem{remark}[thm]{Remark}

\newtheorem{manualtheoreminner}{Theorem}
\newenvironment{manualtheorem}[1]{%
  \renewcommand\themanualtheoreminner{#1}%
  \manualtheoreminner
}{\endmanualtheoreminner}

\usepackage{pdfpages}

\usepackage{setspace}
\usepackage{fancyhdr}

\title{Gibbs Properties of Equilibrium States}
\author{\small{Mirmukhsin Makhmudov$^{1}$ and Evgeny Verbitskiy$^{1,2}$}}
\address{\small{$^1$Mathematical Institute,  Leiden University \\
         $^2$Korteweg-de Vries Institute for Mathematics, University of Amsterdam}}
\email{m.makhmudov@math.leidenuniv.nl, evgeny@math.leidenuniv.nl}

\begin{document}
\begin{abstract}

We consider the problem of equivalence of Gibbs states and equilibrium states for continuous potentials on full shift spaces $E^{\Z}$.
Sinai, Bowen, Ruelle and others established equivalence under various assumptions on the potential $\phi$. At the same time, it is known that every ergodic measure is an equilibrium state for some continuous potential. 
This means that the equivalence 
can occur only under some appropriate conditions on the potential function.
In this paper, we identify the necessary and sufficient conditions for the equivalence.

\end{abstract}

\maketitle
 
\section{Introduction}

DLR Gibbs measures were introduced by Dobrushin (1968) and Lanford and Ruelle (1969) to describe the collective behaviour of a system composed of a large number of components, each governed by a local law.
Soon after, the Gibbs measures found applications in other fields of science and various areas of mathematics. 
In particular, in the early 1970's,
Sinai 
showed that natural invariant measures for hyperbolic dynamical systems are Gibbs measures.
The original definition of Gibbs measures in statistical mechanics is somewhat cumbersome in the context of dynamical systems.
For this reason, Bowen \cite{Bowen-book} provided a more suitable definition of Gibbs states from the dynamical systems perspective: a translation-invariant measure $\mu$ on $\Omega=E^{\Z}$, $E$ is finite, is called \textit{Gibbs in Bowen's sense}, or \textit{Bowen-Gibbs}, for a continuous potential $\phi: \Omega\to \R$, if for some constants $C>1$ and $P$, 
for all $n\in\N$ and every $\omega\in \Omega$ 
\begin{equation}\label{Bowen-Gibbs property}
    \frac{1}{C}
    \leq
    \frac{\mu\big(\{\tilde\omega\in \Omega: \tilde\omega_0^{n-1}=\omega_0^{n-1}\}\big)}{\exp\big(S_n\phi (\omega)-nP\big)}
        \leq C,
\end{equation}
where $S_n\phi(\omega)=\sum_{k=0}^{n-1}\phi(S^k\omega) $ and  $S:\Omega\to \Omega$ is the left shift on $\Omega$.
Subsequently, weaker versions of this notion were introduced. 
Namely, a translation-invariant measure $\mu$ on $\Omega$ is \textit{weak Bowen-Gibbs} if $\mu$ satisfies
\begin{equation}\label{eq: weak Bowen-Gibbs property}
    \frac{1}{C_n}
        \leq
        \frac{\mu(\{\tilde\omega\in \Omega: \tilde\omega_0^{n-1}=\omega_0^{n-1}\})}{e^{S_n\phi (\omega)-nP}}
        \leq C_n,
\end{equation}
for some subexponential sequence $\{C_n\}$ of positive real numbers, i.e., $\lim_{n\to\infty}\frac{1}{n}\log C_n=0$. 

Let us stress that Bowen's definition 
of Gibbs measures is actually a theorem in Statistical Mechanics. 
More specifically, if $\mu$ is a translation-invariant DLR-Gibbs measure 
(see Section \ref{Section DLR Gibbs Formalism} for the notion), then there exists a continuous function $\phi:\Omega\to\R$ and positive numbers $C_n=C_n(\phi)$ such that (\ref{eq: weak Bowen-Gibbs property}) holds.
Therefore, we prefer to use the name of Gibbs measures for measures which are Gibbs in the DLR sense, and we refer to the measures satisfying (\ref{Bowen-Gibbs property}) and (\ref{eq: weak Bowen-Gibbs property}) as Bowen-Gibbs and weak Bowen-Gibbs measures. 

The notion of weak Gibbs states in Bowen's sense is also somewhat misleading, as it suggests some form of non-Gibbsianity and competes with a notion under the same name in Statistical Mechanics. 
As we will see below, weak Bowen-Gibbs measures can be bona-fide Gibbs measures in the DLR sense.
We also note that there are examples of weak Gibbs measures in the DLR sense, which are weak Bowen-Gibbs as well \cite{MRTMV2000}.

Bowen's definition (\ref{Bowen-Gibbs property}) and its weak form (\ref{eq: weak Bowen-Gibbs property}) are extremely convenient from the Dynamical Systems point of view. 
At the same time, using such definitions, one can, in principle, say very little about the conditional probabilities of the underlying measure, which is the classical approach to Gibbs measures in Statistical Mechanics.
In fact, there exist Bowen-Gibbs measures that are not DLR-Gibbs measures \cite{BFV2019}*{Subsection 5.2}.

Another important notion is that of equilibrium states. 
A translation-invariant measure $\mu$ on $\Omega=E^{\Z}$ is called an \textit{equilibrium state} for a (continuous) potential $\phi:\Omega\to\R$ if 
\begin{equation}
    h(\mu)+\int_{\Omega} \phi d\mu =P(\phi),
\end{equation}
where $P(\phi)$ is the topological pressure of $\phi$ and $h(\mu)$ is the measure-theoretic entropy of $\mu$.
For expansive systems like those we consider in this paper, equilibrium states always exist.
It is easy to see that a weak Bowen-Gibbs state is also an equilibrium state for the same potential \cite{PS2018}.

A fundamental result highlighting the breadth of the class of equilibrium states is the following: if $\mu_1, \dots, \mu_k$ are some ergodic measures on $\Omega$,
then one can find a continuous potential $\phi\in C(\Omega)$ such that all these measures are equilibrium states for $\phi$ \cites{Ruelle-book, Israel-book, EFS1993}.
This remarkable generality suggests that equilibrium states can exhibit a wide range of behaviors, and in particular, one cannot expect them to possess any form of Gibbsianity in general.
This leads to a natural question: under what conditions on $\phi$ are the equilibrium states Gibbs, either in the DLR or the Bowen sense?
This question has a long history of research: 
\begin{itemize}
    \item[$\bullet$] For \textit{H\"older continuous} potentials, Sinai proved that equilibrium states are Gibbs in the DLR sense \cite{Sinai1972}*{Theorem 1}, and the Bowen-Gibbs property was established by Bowen \cite{Bowen-book}.
    Haydn extended Sinai's results to the non-symbolic setup \cites{H1987, H1994}.
    
    \item[$\bullet$] Ruelle \cite{R1968} (see also \cite{Keller-book}*{Theorem 5.3.1}) 
    studied the DLR Gibbsianity of the unique equilibrium states for  potential $\phi$ with \textit{summable variations}: 
    $$\sum_{n\geq 1}\mathtt{var}_{n}\phi <\infty,\quad 
    \mathtt{var}_{n}\phi:=\sup\big\{\phi(\omega)-\phi(\bar\omega): \omega_j=\bar\omega_j,\;\; 0\leq j\leq n-1\big\}.
    $$ 
    The Bowen-Gibbs property was treated by Keller \cite{Keller-book}*{Theorem 5.2.4, (c)} 
    
    \item[$\bullet$] Walters \cite{Walters1978} considered potentials satisfying even a weaker condition: 
    \begin{equation}\label{Walters condition}
        \lim\limits_{p\to\infty}\sup_{n\in\N}\mathtt{var}_{[-p, n+p]} S_{n+1}\phi=0,\;\;\text{ here }\;\;  S_{n+1}\phi=\sum_{i=0}^n\phi\circ S^i.
    \end{equation}
    Walters showed that there exists a unique equilibrium state for potentials satisfying (\ref{Walters condition}), and they have the so-called, $g-$measure property, which amounts to saying that the equilibrium state has continuous one-sided conditional probabilities.
    Combining this with the result of \cite{BFV2019}, one concludes the DLR Gibbsianity of the unique equilibrium state of a potential in the Walters class.
    In
    \cite{HR1992}, Haydn and Ruelle extended this to a more general setup than the setup of shift spaces.
    In the same paper, Haydn and Ruelle also established the Bowen-Gibbs property of the unique equilibrium state for a potential satisfying the \textit{Bowen condition}: 
    $$
    \sup_{n\in\N}\texttt{var}_{[-n,n]}S_{n+1}\phi<+\infty,
    $$ 
    which is slightly weaker than Walters' original condition.
    
    \item[$\bullet$] More recently, Pfister and Sullivan \cite{PS2020} established the weak Bowen-Gibbs property of equilibrium states for potentials $\phi$ with \textit{summable oscillations}: 
    $$\sum_{i=-\infty}^\infty\delta_i\phi<+\infty,\quad \delta_i\phi:=\sup\{\phi(\omega)-\phi(\bar\omega): \omega_j=\bar\omega_j,\;\; j\neq i\}.
    $$
    Unlike the preceding conditions, the summable oscillations condition does not imply the uniqueness of the corresponding equilibrium states.
    However, the DLR Gibbs property of  equilibrium states under summable oscillations has not been addressed. 
\end{itemize}

This paper continues the long line of research on the Gibbsianity of equilibrium states, as discussed above, and also extends the main result of \cite{BFV2019}, where a similar question has been answered in the case of $g-$measures.
In this paper, we show that under a similar assumption on the potential $\phi$, one can establish the Gibbs properties of equilibrium states and vice versa.
This assumption on the regularity of the potential $\phi: \Omega\to\R$ is the \textit{extensibility condition},
which requires that for all $a_0,b_0\in E$ the sequence
of functions 
 $$
 \rho_n^{a_0,b_0}(\omega)
 :
 =
 \sum_{i=-n}^n \bigl(\phi\circ S^{i}(\omega_{-\infty}^{-1}b_0\omega_1^{\infty})-\phi\circ S^{i}(\omega_{-\infty}^{-1}a_0\omega_1^{\infty})\bigr)
 $$
converges uniformly in $\omega\in \Omega$ as $n\to\infty$. 
The extensibility condition is not very restrictive.
For example, it does not imply the uniqueness of the equilibrium states, unlike the results by Sinai, Bowen, Ruelle and Walters.
Furthermore, the extensibility condition covers the previously treated classes, including
the class of H\"older continuous potentials, potentials with summable variations, Walter's class, as well as the class of potentials with summable oscillations. 
However, the potentials in Bowen's class do not necessarily have the extensibility property \cite{BFV2019}*{Section 5.5}. 
An important example of extensible potentials is the Dyson potential, $\phi^D(\omega):=h\omega_0+\sum_{n=1}^\infty\frac{\beta \omega_0\omega_n}{n^\alpha},\; \omega\in \{\pm1\}^\Z$, which has been extensively 
studied recently  \cites{JOP2023, EFMV2024, M2025}, where $h,\beta\in\R$ and $\alpha>1$.

The following theorem, the first of our main results in this paper, establishes the Gibbs properties of the equilibrium states of an extensible potential.

\begin{manualtheorem}{A}\label{Th A: ES for extensible potentials are DLR and Bowen-Gibbs}
    Suppose $\phi\in C(\Omega)$ has the extensibility property. 
    Then any equilibrium state $\mu\in \es(\phi)$ is
    \begin{itemize}
        \item[(1)] Gibbs in the Dobrushin-Lanford-Ruelle sense;
        \item[(2)] weak Bowen-Gibbs relative to the potential $\phi$.
    \end{itemize}
\end{manualtheorem}

The proof of Theorem \ref{Th A: ES for extensible potentials are DLR and Bowen-Gibbs} is given in Section \ref{proofs of main results}, and uses the following idea:
for a given potential $\phi\in C(\Omega)$ satisfying the extensibility condition, we construct a natural two-sided Gibbsian specification (a consistent family of regular probability kernels) $\gamma^\phi$ on $\Omega=E^\Z$.
Then we show that any translation-invariant DLR-Gibbs state $\nu$ for the specification $\gamma^\phi$ will be an equilibrium state for $\phi$.
Hence, the set of Gibbs states associated with the specification $\gamma^\phi$ is a subset of the set of equilibrium states for $\phi$.
Finally, if we take any equilibrium state $\tau\in \es(\phi)$ and a DLR-Gibbs state $\mu\in \G_S(\gamma^\phi)$, we will show that the relative entropy density rate $h(\tau|\mu)$ is zero.
This allows us to use the classical variational principle. 
\cite{KLNR2004}*{Theorem 4.1} and conclude that $\tau\in \G_S(\gamma^\phi)$ as well.

Our second main result is about the translation-invariant Gibbs measures, which in some sense is a converse of Theorem A.
\begin{manualtheorem}{B}\label{Th B: Gibbs states are equilibrium for extensible phi}
    Assume $\mu$ is a translation-invariant DLR Gibbs measure on $\Omega=E^\Z$.
    Then $\mu$ is an equilibrium state for a 
    potential
    with the extensibility property.
\end{manualtheorem}
We should note that if $\mu$ is a Gibbs measure for a translation-invariant \textit{uniformly absolutely convergent} (UAC) interaction, then the claim is rather standard \cite{LR1969}*{Theorem 3.2}.
However, as demonstrated in \cite{BGMMT2020}, not all translation-invariant Gibbs measures are compatible with a translation-invariant UAC interaction.
Thus, Theorem B generalizes the result in \cite{LR1969}*{Theorem 3.2} to a broader setting, encompassing all translation-invariant Gibbs measures, including those that are not Gibbs for any translation-invariant UAC interaction.

The proof of Theorem \ref{Th B: Gibbs states are equilibrium for extensible phi} is constructive and is also given in Section \ref{proofs of main results}.
In fact, we construct a natural one-sided potential $\phi_\gamma$ out of the Gibbsian specification $\gamma$ for $\mu$.
Then we show that $\phi_\gamma$ is extensible, and this allows us to apply Theorem \ref{Th A: ES for extensible potentials are DLR and Bowen-Gibbs} to $\phi_\gamma$.

One might observe an analogy between our results and Sullivan's theorem \cite{Sullivan1973}*{Theorem 1} in Statistical Mechanics. 
In Sullivan's theorem, the role of extensible potentials is played by the so-called $\mathscr{L}$-convergent interactions,  a notion that is also syntactically similar to the notion of extensibility.

We also note that the statements of Theorems A and B, along with their proofs presented in this paper, naturally extend to higher-dimensional lattices $\Z^d$ ordered lexicographically.
Since there is no substantial difference in the proofs, we only focus on the one-dimensional lattice $\Z$.

\begin{wrapfigure}{r}{6.9cm} 
\centering

\begin{tikzpicture}[scale=0.75]
    \fill[gray!10, fill opacity=0.1] (-1.5,-1.34) ellipse (4.4 and 3.5);   
    \draw[black!50] (-1.5,-1.34) ellipse (4.4 and 3.5);
    \node[left, black!70] at (0.8,1.3) {\textbf{\footnotesize{Equilibrium states}}};
    
    \fill[gray!20, fill opacity=0.1] (-1.5,-1.6) ellipse (3.8 and 2.55);   
    \draw[black!50] (-1.5,-1.6) ellipse (3.8 and 2.55);
    \node[right, black!78] at (-4.8,-2.8) {\textbf{\footnotesize Weak Bowen-Gibbs states}};
    
    \begin{scope}
        \clip (-1.5,-1.5) ellipse (2.5 and 0.9);
        \fill[gray!30, fill opacity=0.15] (-1.5,-0.25) ellipse (2.5 and 0.9);
    \end{scope}
    
    \begin{scope}
        \draw[black!50] (-1.5,-1.5) ellipse (2.5 and 0.9);
    \end{scope}
    
    \begin{scope}
        \draw[black!50] (-1.5,-0.25) ellipse (2.5 and 0.9);
    \end{scope}
    
    \node[below, black!85] at (-1.5,-1.15) {\textbf{\footnotesize{Bowen-Gibbs states}}};
    \node[below, black!99] at (-1.4,0.32) {\textbf{\footnotesize DLR Gibbs states}};
\end{tikzpicture}
    \label{fig:venn_gibbs states}
\end{wrapfigure}
\noindent

The diagram of the right summarizes the relationship between equilibrium states and various notions of Gibbs states.
\bigskip

\noindent
The paper is organized as follows:
\begin{itemize}[leftmargin=0.2cm]
    \item In Section \ref{Section DLR Gibbs Formalism}, we introduce the basic concepts in the DLR Gibbs formalism such as Gibbs measures, specifications and interactions.
    Here, we also recall some important results, such as the variational principle from the classical DLR Gibbs formalism.
\end{itemize}    

    \WFclear 
    
\begin{itemize}[leftmargin=0.2cm]
    \item In Section \ref{Section Bowen's property of Gibbs measures}, we discuss the motivation behind Bowen's definition of Gibbs states.
    \item In Section \ref{Potentials, Cocycles, and Specifications}, we discuss the relationship between the extensible potentials and Gibbsian specifications.
    \item Section \ref{proofs of main results} is dedicated to the proofs of the main results in this paper.
\end{itemize}

\begin{center}
\section{DLR Gibbs Formalism}\label{Section DLR Gibbs Formalism}
\end{center}

The theory of Gibbs states, which is put forward by Dobrushin, Lanford, and Ruelle, is very flexible and allows one to define Gibbs states on very general lattice spaces $E^\bL$, where $E$ is a Polish space and $\bL$ is a countable set.
In the present paper, we are primarily interested in probability measures
on $\Omega=E^{\Z}$, $E$ is finite, which are invariant under the left shift $S:\Omega\to\Omega$.

\subsection{Specifications, Interactions, and Gibbs states in Statistical Mechanics}\label{StatMech}
The standard Statistical Mechanics description of Gibbs states is rather different from the
definitions of Bowen-Gibbs and weak Bowen-Gibbs measures. 
The principal point is the explicit description of the family of \mybf{conditional expectations} 
indexed by finite subsets $\Lambda$ of $\Z$.  
More precisely, in Statistical Mechanics, one starts with a family of regular conditional expectations, which for $f:\Omega\to\R$, given by
$$
\gamma_\Lambda (f|\omega)
:=
\sum_{\xi_{\Lambda}\in E^{\Lambda}}
\gamma_{\Lambda}(\xi_\Lambda| \omega_{\Lambda^c}) 
f(\xi_\Lambda \omega_{\Lambda^c}),
$$
where 
\begin{equation}\label{eq: Boltzmann ansatz first time}
\aligned
\gamma_{\Lambda}(\xi_\Lambda |\omega_{\Lambda^c})&=
\frac 
{\exp\Bigl( -H_{\Lambda}(\xi_\Lambda \omega_{\Lambda^c})\Bigr)}
{Z_\Lambda(\omega)},\quad
{Z_\Lambda(\omega)}=\sum_{
\zeta_{\Lambda}\in E^{\Lambda}
}
\exp\Bigl( -H_{\Lambda}(\xi_\Lambda \omega_{\Lambda^c})\Bigr).
\endaligned
\end{equation}
In order to guarantee the tower property of the conditional expectations, one needs to assume 
\textit{consistency} of $\gamma_\Lambda$'s: $\gamma_\Lambda=\gamma_\Lambda\circ \gamma_V$ for $V\subset \Lambda$, where for $f:\Omega\to\R$ measurable and $\omega\in\Omega$, $(\gamma_\Lambda\circ \gamma_V)(f|\omega):=\int_{\Omega}\gamma_V(f|\eta)\gamma_\Lambda(d\eta|\omega)$. 
The latter is ensured if the functions $H_{\Lambda}:\Omega\to\R$ -- called a \textit{Hamiltonian} in $\Lambda$-- are of a rather special form:
\begin{equation} \label{Hamiltonian}   
H_{\Lambda}(\omega) = \sum_{\substack{V\Subset \Z\\ V\cap \Lambda\ne\varnothing}}\Phi_V(\omega_V),
\end{equation}
here the summation is taken over all finite subsets of $\Z$ (denoted by $\Subset \Z$) which have
non-empty intersection with $\Lambda$.
Here $\Phi=\{\Phi_V, V \Subset \Z\}$ is called an \mybf{interaction} and each function $\Phi_V: \Omega\to\R$ is \mybf{local}, meaning that the value of $\Phi_V(\omega)$ depends only on the values of $\omega$ within $V$, hence, we write $\Phi_V(\omega_V)$.
In order for these expressions to make sense, one needs to assume a suitable 
form of summability in \eqref{Hamiltonian}. 
The standard and sufficient assumption is \mybf{uniform absolute convergence} (UAC): for all $i\in \Z$,
$$
\sum_{i\in V\Subset \Z} \| \Phi_V\|_{\infty}=\sum_{i\in V\Subset \Z} \sup_{\omega\in\Omega}
| \Phi_V(\omega_V)|<\infty.
$$
If $\Phi$ is an UAC interaction, the corresponding \textit{specification} $\gamma=(\gamma_\Lambda)_{\Lambda\Subset\Z}$ defined by (\ref{eq: Boltzmann ansatz first time}), is \begin{itemize}
    \item \textit{non-null}: for all $\Lambda$, $\inf_\omega\gamma_\Lambda(\omega_\Lambda|\omega_{\Lambda^c})>0$,
    \item \textit{continuous}: for every $\Lambda$, $\omega\mapsto \gamma_\Lambda(\omega_\Lambda|\omega_{\Lambda^c})$ is continuous.
    \end{itemize}
In statistical mechanics, the second property is often referred to as \emph{quasi-locality}.
An important property of positive specifications, which will be used in the proofs, is the so-called \textit{bar moving property}: 
\begin{equation}\label{eq: bar moving prop}
   \frac{\gamma_{\Delta}(\xi_\Delta|\omega_{\Delta^c})}{\gamma_{\Delta}(\zeta_\Delta|\omega_{\Delta^c})}
   =
   \frac{\gamma_{\Lambda}(\xi_\Delta\omega_{\Lambda\setminus\Delta}|\omega_{\Lambda^c})}{\gamma_{\Lambda}(\zeta_\Delta\omega_{\Lambda\setminus\Delta}|\omega_{\Lambda^c})}, 
\end{equation}
for all $ \xi,\zeta,\omega\in\Omega$ and every $\Delta\subset\Lambda\Subset\Z$.
The bar moving property is equivalent to the consistency condition of specifications.

A non-null  continuous specification  $\gamma=\{\gamma_\Lambda\}$ is called \textit{\textbf{Gibbsian}}.

\begin{defin}\label{def of DLR Gibbs measures based on spec}
Suppose $\gamma=\{\gamma_\Lambda\}$ is a Gibbsian specification on $\Omega$. 
The measure $\mu$ is called Gibbs for  $\gamma$,
denoted by $\mu\in\mathcal G(\gamma)$, 
if for every $\Lambda\Subset\Z$,
$$
\mu(\omega_\Lambda|\omega_{\Lambda^c}) = \gamma_\Lambda(\omega_\Lambda|\omega_{\Lambda^c}), \; \text{ for } \mu-\text{a.e. } \omega\in\Omega,
$$
equivalently, if the DLR equations hold: for every $f\in C(\Omega)$ and $\Lambda\Subset\Z$, 
\begin{equation*}
    \int_\Omega \gamma_\Lambda(f|\omega)\mu(d\omega)=\int_\Omega f(\omega)\mu(d\omega).
\end{equation*}
\end{defin}

For any Gibbsian specification $\gamma$, the set of corresponding Gibbs measures
$\mathcal G(\gamma)$ is a non-empty convex set. In case, $\mathcal G(\gamma)$
consists of multiple measures, one says that $\gamma$ exhibits phase transitions.

\subsection{Translation-invariance and the Variational Principle}

The modern approach to Gibbsian formalism is to think about Gibbsian specifications $\gamma$
in an \mybf{interaction independent} fashion. 
The reason for this is that a measure
$\mu$ can be consistent with at most one Gibbsian specification, while there are infinitely many interactions $\Phi$ giving
rise to the same Gibbsian specification.
In fact, it was proven by Kozlov that for any Gibbsian specification $\gamma$ there exists a (in fact many) UAC interaction $\Phi$ such that $\gamma=\gamma^\Phi$.
However, such a representation is not always possible in a way that respects translation invariance.
It is shown in \cite{BGMMT2020} that there exists a Gibbsian specification $\gamma$ on $\{0,1\}^\Z$ that is \textit{translation-invariant} -- meaning $\gamma_{\Lambda+1}=\gamma_\Lambda\circ S$ for every $\Lambda\Subset\Z$ -- but can not be associated with any \textit{translation-invariant UAC interaction} $\Phi$, where $\Phi_\Lambda=\Phi_\Lambda\circ S$ for all $\Lambda\Subset\Z$, via (\ref{eq: Boltzmann ansatz first time}).
Nevertheless, many important results in Statistical Mechanics, including the \textit{variational principle}, can be formulated independently of interactions. 
\begin{thm}\cite{KLNR2004}\label{VP by KLNR}
    Let $\gamma$ be a translation-invariant Gibbsian specification and $\mu\in\G_S(\gamma)$.
    Then for all $\tau\in\M_{1,S}(\Omega)$, the specific relative entropy $h(\tau|\mu)$ exists and
    $$
    h(\tau|\mu)=0\;\; \Longleftrightarrow\;\; \tau\in\G_S(\gamma).
    $$
\end{thm}
In \cite{KLNR2004}, the authors proved some technical lemmas about translation-invariant specifications, which would be useful for us.
For $n\in\N$, we set $\Lambda_n:=[-n,n]\cap\Z$.
We fix a letter $\la$ in the finite alphabet $E$ and we will use the notation $\bla$, to denote the constant configuration consisting of $\la$'s.
\begin{lemma}\label{lemma: pressure of specification and variation of spec}
Suppose $\gamma$ is a translation-invariant Gibbsian specification on $\Omega$. Then
    \begin{itemize}
        \item[(i)]
        \begin{equation}\label{eq: est about decay rate of spec ratios}
        \sup_{\sigma, \omega, \eta\in\Omega} \Big\lvert
        \log 
        \Big(
        \frac{\gamma_{[0,n]}(\sigma_{[0,n]}|\omega_{[0,n]^c})}
        {\gamma_{[0,n]}({\bla}_{[0,n]}|\omega_{[0,n]^c})} 
        \cdot
        \frac{\gamma_{[0,n]}({\bla}_{[0,n]}|\eta_{[0,n]^c})}
        {\gamma_{[0,n]}(\sigma_{[0,n]}|\eta_{[0,n]^c})}
        \Big)
        \Big\rvert
        =
        o(n).
\end{equation}
        \item[(ii)] the sequence $\Big\{-\frac{1}{|\Lambda_n|} \log \gamma_{\Lambda_n}(\bla_{\Lambda_n}|\omega_{\Lambda_n^c}) \Big\}_{n\in\N}$ converges uniformly in $\omega\in \Omega$ to a constant which we denote by $P^\bla(\gamma)$.
    \end{itemize}
\end{lemma}

\section{Bowen's property of Gibbs measures}\label{Section Bowen's property of Gibbs measures}

By comparing the definitions of Gibbs measures in Statistical Mechanics and that of Bowen, it is immediately clear why Bowen's definition is so attractive
and popular for dynamicists: it captures the most important, from the Dynamical Systems point of view, properties of Gibbsian states -- uniform estimates on measures of cylindric sets in terms of ergodic averages of the potential function.
In fact, whether the
measure has the Bowen property or the weak Bowen property, rarely makes any difference in Dynamical Systems: the subexpontential bound is as good as the uniform bound in practically any computations.

Nevertheless, Bowen was fully aware that his definition of Gibbs states is not the same as in Statistical Mechanics: {\it "In statistical mechanics, Gibbs states are not defined by the above
theorem. We have ignored many subtleties that come up in more complicated systems"}, \cite{Bowen-book}*{page 6}.
To introduce the definition, Bowen was motivated by an example (\cite{Bowen-book}*{page 5}) of a translation-invariant pair interaction $\Phi$, $\Phi_{V}\not\equiv 0$ only if $V=\{k\}$ or $V=\{k,n\}$, $k,n\in\Z$, satisfying a strong summability condition
\begin{equation}
    \|\Phi_{\{0\}}\|_\infty+\frac 12\sum_{n\in\Z\setminus\{0\}} |n|\cdot \|\Phi_{\{0,n\}}\|_\infty<\infty.
\end{equation}
The above  condition is a special case of a well-known uniqueness condition in thermodynamic formalism \cites{Ruelle-book, Georgii-book}:
\begin{equation}\label{diam-UAC condition}
    \sum_{0\in V\Subset \Z}\frac{\diam(V)}{|V|}\cdot\|\Phi_V\|_{\infty}<\infty.
\end{equation}
Let us now discuss the Bowen-Gibbs and the weak Bowen-Gibbs properties of DLR Gibbs states.
\begin{thm}\label{thm:BowenProperty} Suppose $\Phi=\{\Phi_V\}_{V\Subset \Z}$
is a translation-invariant UAC interaction and  let
$\phi=-\sum_{0\in V\Subset\Z_+}\Phi_V$.
Then there exists a sequence $\{C_n\}$ with $n^{-1}\log C_n\to 0$,  such that for every translation-invariant Gibbs measure $\mu$ for $\Phi$, for all $n$ and $\omega\in \Omega$, one has
$$
\frac 1{C_n}\le \frac{ \mu\bigl(\{\tilde\omega\in \Omega:\ \tilde\omega_{0}^{n-1}=\omega_{0}^{n-1}\}\bigr)}
{
\exp\bigl( S_n\phi(\omega) -nP(\phi)\bigr)
}
\le C_n.
$$
If, furthermore, the interaction $\Phi$ satisfies a stronger summability condition (\ref{diam-UAC condition}),
then there exists a unique Gibbs measure $\mu$ for $\Phi$, and
for some $C>1$, every $n\ge 1$ and all $\omega\in\Omega$,
\begin{equation}\label{Bowen-Gibbs prop under strong summability of interac}
\frac 1{C}\le \frac{ \mu(\{\tilde\omega\in \Omega:\ \tilde\omega_{0}^{n-1}=\omega_{0}^{n-1}\})}
{
\exp\bigl( S_n\phi(\omega) -nP(\phi)\bigr)
}
\le C.    
\end{equation}

\end{thm}
Let us sketch the proof of this theorem using known results in Statistical Mechanics.
The first claim is standard \cite{Georgii-book}*{Theorem 15.23}. Applying the DLR equations to the
indicator function of the cylinder set $[\sigma_0^{n-1}]$, one concludes that
$$
\mu([\sigma_0^{n-1}])=\int \frac { \exp\big(-H_{\Lambda_n}(\sigma_{\Lambda_n}\eta_{\Lambda_n^c})\big)}{Z_{\Lambda_n}(\eta)} \mu(d\eta).
$$
The sequence of functions $\frac{1}{|\Lambda_n|}\log Z_{\Lambda_n}(\eta)$ converges to the pressure $P(\Phi)=P(\phi)$ uniformly in $\eta$ as $n\to\infty$ \cite{Georgii-book}*{Theorem 15.30, part (a)}. 
Thus $\log Z_{\Lambda_n}(\eta)=\exp( |\Lambda_n|P +o(n))$.

To study the numerator, we use the estimate  (15.25) in \cite{Georgii-book}:
\begin{equation}\label{uni est for diff of ergod sum and Hamilt}
    \sup_{\sigma, \eta\in\Omega}
    \Big\lvert
    \sum_{i\in\Lambda_n}\phi \circ S^i(\sigma)
    +
    H_{\Lambda_n}(\sigma_{\Lambda_n}\eta_{\Lambda_n^c})
    \Big\rvert
    \leq \sum_{i\in\Lambda_n}\sum_{\substack{V\ni i\\ V\not\subset\Lambda_n}}
    \rVert\Phi_V\lVert_\infty,
\end{equation}
and the fact that the uniformly absolute convergence of the interaction $\Phi$ ensures that 
\begin{equation*}
    \sum_{i\in\Lambda_n}\sum_{\substack{V\ni i\\ V\not\subset\Lambda_n}}
    \rVert\Phi_V\lVert_\infty\;
    =
    o(|\Lambda_n|).
\end{equation*}
The uniqueness of the Gibbs measures under (\ref{diam-UAC condition}) follows from Theorem 8.39 in \cite{Georgii-book} (see also Comment 8.41 and equation (8.42) in \cite{Georgii-book}).
Note that (\ref{diam-UAC condition}) implies that $\sum_{i\in\Lambda_n}\sum_{\substack{V\ni i\\ V\not\subset\Lambda_n}}
    \rVert\Phi_V\lVert_\infty$
remains bounded as $n\to\infty$, in fact, for all $n\in\N$, one has that 
$$
\sum_{i\in\Lambda_n}\sum_{\substack{V\ni i\\ V\not\subset\Lambda_n}}
\rVert\Phi_V\lVert_\infty
\leq
\sum_{0\in V\Subset\Z_+} {\diam (V)}\cdot \| \Phi_V\|_\infty
=
\sum_{0\in V\Subset\Z}\frac{\diam(V)}{|V|}\cdot \|\Phi_V\|_\infty
=:D.
$$
This together with (\ref{uni est for diff of ergod sum and Hamilt}) yields that 
\begin{equation}
  \sup_{\sigma, \eta\in\Omega}
    \Big\lvert
    \sum_{i\in\Lambda_n}\phi \circ S^i(\sigma)
    -
    \sum_{i\in\Lambda_n}\phi \circ S^i(\sigma_{\Lambda_n}\eta_{\Lambda_n^c})
    \Big\rvert\leq 2D,  
\end{equation}
i.e., $\phi$ satisfies Bowen's condition \cite{Walters2001}.
Then (\ref{Bowen-Gibbs prop under strong summability of interac}) follows from Theorem 4.6 in \cite{Walters2001} and from the fact that the translation-invariant Gibbs measures for $\Phi$ are equilibrium states for $\phi$ \cite{Georgii-book}*{Theorem 15.39}.

\begin{remark} Note that it is not a coincidence that 
the condition (\ref{diam-UAC condition}) implying the Bowen property, also implies uniqueness.
Indeed, suppose $\mu$ and $\nu$ are two ergodic measures on $\Omega$ with the Bowen-Gibbs property for some continuous potential $\phi$. Then the Bowen-Gibbs  property implies
that 
$$
\frac 1C \le\frac {\mu([\omega_0^n])}{\nu([\omega_0^n])}\le C
$$
for all $n$ and every $\omega\in\Omega$, and hence the measures $\mu$ and $\nu$ are equivalent,
and thus are equal. 
\end{remark}

\section{Potentials, Cocycles, and Specifications}\label{Potentials, Cocycles, and Specifications}
An alternative, more dynamical approach to DLR-Gibbs measures, also known as the Ruelle-Capocaccia approach, was introduced in \cite{Capocaccia1976}. 
However, as demonstrated in the original work \cite{Capocaccia1976}, for lattice systems—which is the setting in this paper—the Ruelle-Capocaccia definition of Gibbs measures coincides with the specification-based definition given in
this paper.
Keller’s book \cite{Keller-book}*{Chapter 5} provides an excellent summary of Gibbs measures following the Ruelle-Capocaccia approach under the assumption that the underlying potential has summable variations.
Below we explore how this approach extends to cases where the potential lacks summable variations but still has the extensibility property. 

Recall that the extensibility condition requires that for all $\omega\in\Omega$ and every $a,\tilde a\in E$, the sequence
of functions
\begin{equation}\label{cocycle for extensible potential_finite n 0 site case}
    \rho_n^{a,\tilde a }(\omega)
=\sum_{i=-n}^n \big[\phi(S^i\omega^a)-\phi(S^i\omega^{\tilde a})\big], \quad n\ge 0,
\end{equation}
converges uniformly as $n\to\infty$. 
Here, we use the notation $\omega^a=(\omega^a_k)_{k\in\Z}$ is given by
$$
\omega^a_k=\begin{cases} a,&\ k=0,\\
\omega_k,&\ k\ne 0.	
\end{cases}
$$
Therefore, we can define a continuous function $\rho_n^{a,\tilde a }(\omega):\Omega\to\R$,
$$
\rho(\omega^a, \omega^{\tilde a})=\lim_{n\to\infty} \rho_n^{a,\tilde a }(\omega).
$$

\begin{prop}
    Suppose $\phi$ satisfies the extensibility condition, then for any pair $\xi,\eta\in\Omega$,
such that the set $\{k\in\Z: \xi_k\ne\eta_k\}$ is finite, the sequence of continuous functions
\begin{equation}\label{cocycle for extensible potential_finite n}
    \rho_n(\xi,\eta)= \sum_{i=-n}^n \big[\phi(S^i\xi)-\phi(S^i\eta)\big], \quad n\ge 0,
\end{equation}
converges.
Furthermore,
\begin{itemize}
    \item[(1)] the limiting function $\rho(\xi,\eta)=
    \lim_n\rho_n(\xi,\eta)$ is a \textbf{cocycle}, i.e.,
    for every $\xi, \eta, \zeta\in \Omega$ with $\xi_i=\eta_i=\zeta_i$ for all $i$ with $|i|\gg1$, one has
    \begin{equation}\label{cocycle equation}
    \rho(\xi,\zeta)=\rho(\xi,\eta)+\rho(\eta, \zeta);
    \end{equation}
    
    \item[(2)] $\rho$ is translation-invariant in the sense that for every pair $(\xi,\eta)$ with $\{k\in\Z: \xi_k\ne\eta_k\}$ finite, 
    \begin{equation}\label{tr-inv of the cocycle}
    \rho(\xi, \eta)=\rho(S\xi, S\eta).
    \end{equation}

    \item[(3)] for every $\Lambda\Subset\Z$ and $\eta_\Lambda, \zeta_\Lambda\in E^\Lambda$, $\rho(\eta_\Lambda \xi_{\Z\setminus\Lambda}, \zeta_\Lambda \xi_{\Z\setminus\Lambda})$ is a continuous function of $\xi$.
\end{itemize}
\end{prop}
\begin{proof}
We carry out the proof in two steps.

\noindent
\textit{First Step:} Let $\xi,\eta\in\Omega$ such that for some $k\in\Z$, $\xi_{\Z\setminus\{k\}}=\eta_{\Z\setminus\{k\}}$.
Without loss of generality, let $k> 0$.
Then 
\begin{eqnarray}
    \sum_{i=-n}^n [\phi\circ S^{i-k}(S^{k}\xi)-\phi\circ S^{i-k}(S^{k}\eta)]
    \notag &=&
    \sum_{i=-n-k}^{n-k} [\phi\circ S^{i}(S^{k}\xi)-\phi\circ S^{i}(S^{k}\eta)]\\
    &=&\label{extensible applied part for shifted config.s}
    \sum_{i=-n}^{n}
    [\phi\circ S^{i}(S^{k}\xi)-\phi\circ S^{i}(S^{k}\eta)]\\
    &-&\label{- leftover terms}
    \sum_{i=n-k+1}^{n}
    [\phi\circ S^{i}(S^{k}\xi)-\phi\circ S^{i}(S^{k}\eta)]\\
    &+&\label{+ leftover terms}
    \sum_{i=-n-k}^{-n-1}
    [\phi\circ S^{i}(S^{k}\xi)-\phi\circ S^{i}(S^{k}\eta)].
\end{eqnarray}
Since $(S^k\xi)_{\Z\setminus\{0\}}=(S^k\eta)_{\Z\setminus\{0\}}$, the extensibility property of $\phi$ yields that the sum in (\ref{extensible applied part for shifted config.s}) converges uniformly to $\rho^\phi(S^k\xi, S^k\eta)$ as $n\to\infty$.
Note that for any $i\in\Z$, 
$$
|\phi\circ S^{i}(S^k\xi)-\phi\circ S^{i}(S^k\eta)|
\leq
\delta_{i}\phi
\leq
\mathtt{var}_{(-|i|,\; |i|)}\phi
\xrightarrow[|i|\to\infty]{}0.
$$
Thus for (\ref{- leftover terms}) and (\ref{+ leftover terms}), one has that 
$$
\Big|
\sum_{i=n-k+1}^n
[\phi\circ S^{-i}(S^k\xi)-\phi\circ S^{-i}(S^k\eta)]
\Big|
\leq
k \cdot\mathtt{var}_{(k-n,\; n-k)} \phi
$$
and 
$$
\Big|
\sum_{i=-n-k}^{-n-1}
    [\phi\circ S^{-i}(S^k\xi)-\phi\circ S^{-i}(S^k\eta)]
\Big|
\leq
k \cdot\mathtt{var}_{(k-n,\; n-k)} \phi.
$$
Thus since $k$ is fixed and $\mathtt{var}_{(k-n,\; n-k)} \phi
\xrightarrow[n\to\infty]{} 0$, both sums in (\ref{- leftover terms}) and (\ref{+ leftover terms}) converge uniformly to $0$ as $n\to\infty$.
Therefore, $\rho$ is defined at $(\xi, \eta)$ and $\rho(\xi, \eta)=\rho(S^k\xi, S^k\eta)$.

\medskip

\noindent
\textit{Second Step:}
Let $\xi, \eta\in\Omega$ such that for some $\Lambda\Subset\Z$, $\xi_{\Z\setminus\Lambda}=\eta_{\Z\setminus\Lambda}$.
Then there exists $m\in\N$ such that $\Lambda\subset[-m,m]\cap\Z$.
Then for $n>m$, one has
\begin{equation*}
    \sum_{i=-n}^n
    [\phi \circ S^{i}(\xi)-\phi \circ S^{i}(\eta)]
    =
    \sum_{j=-m}^m\sum_{i=-n}^{n}
    [\phi\circ S^{i}(\xi_{-\infty}^{-m-1}\eta_{-m}^{j-1}\xi_{j}^\infty)
    -
    \phi\circ S^{i}(\xi_{-\infty}^{-m-1}\eta_{-m}^{j}\xi_{j+1}^\infty)].
\end{equation*}
By the first step, the sum over $i$ on the RHS of the last equation above converges uniformly for each $j$ as $n\to\infty$.\\
We now address statements (1)-(3).
Claim (1) follows directly from the definition of $\rho$.
Equations (\ref{cocycle equation}) and (\ref{tr-inv of the cocycle}) easily follow from the first and second steps discussed above.
The continuity of the map $\xi \mapsto \rho(\eta_\Lambda \xi_{\Z\setminus\Lambda}, \zeta_\Lambda \xi_{\Z\setminus\Lambda})$ for all $\Lambda \Subset \Z$ and $\eta, \zeta \in \Omega$ is a consequence of the uniform convergence of (\ref{cocycle for extensible potential_finite n}).

\end{proof}

Now, we shall discuss how to associate a Gibbsian specification with an extensible potential and vice versa. 
We have the following theorem.

\begin{thm}\label{defs of associated two-sided spec and extensible potential} (i) Suppose $\phi:\Omega\to\R$ is a continuous function with extensibility property, then 
$\gamma^\phi=(\gamma
^\phi_\Lambda)_{\Lambda\Subset\Z}$ given by 
\begin{equation}\label{eq: def of two sided spec}
    \gamma_\Lambda^\phi(\omega_\Lambda|\omega_{\Lambda^c})=
    \Big(
    \sum_{\xi_\Lambda\in E^\Lambda} 
    e^
    {
    \rho^\phi(\xi_\Lambda\omega_{\Z\setminus\Lambda},\; \omega)
    }  
    \Big)^{-1},\;\; \omega\in\Omega
\end{equation}
is a translation invariant Gibbsian specification.

(ii) Suppose $\gamma=(\gamma_\Lambda)_{\Lambda\Subset\Z}$ is a  translation invariant Gibbsian specification, then
\begin{equation}
   \phi_\gamma(\omega)=\log\frac{\gamma_{\{0\}}(\omega_0| \bla_{-\infty}^{-1}\omega_{1}^\infty)}
   {\gamma_{\{0\}}(\bla_0| \la_{-\infty}^{-1}\omega_{1}^\infty)},\;\; \omega\in\Omega 
\end{equation}
is a continuous function with extensibility property such that $\gamma^{\phi_\gamma}=\gamma$.
\end{thm}
\begin{proof}
\textbf{\textit{{(i):}}}
The association between the Gibbsian specifications and the continuous cocycles given by (\ref{eq: def of two sided spec}) is known and holds in much greater generality than the setup of this paper \cite{BGMMT2020}*{Proposition 2.12 and Proposition 2.13}.
Furthermore, $\gamma^\phi$ is translation-invariant if and only if $\rho^\phi$ is translation-invariant, and the latter has been shown in the above argument (see also Subsection 2.2 and 2.3 in \cite{BGMMT2020}).

\bigskip

\textbf{\textit{{(ii):}}}
Pick any $a_0, b_0\in E$ and $\omega\in\Omega$.
Since $\phi_\gamma$ is a one-sided function, one can check that for any $n\in\N$,
\begin{equation}
    \rho_n^{a_0,b_0}(\omega)-\rho_0^{a_0,b_0}(\omega)
    =
    \sum_{i=1}^n\Big[\phi_{\gamma}\circ S^{-i}(b_0\omega_{\{0\}^c})-\phi_{\gamma}\circ S^{-i}(a_0\omega_{\{0\}^c}) \Big]
\end{equation}
and for the right-hand side of the above equation, one has
\begin{equation}
    RHS
    =
    \sum_{i=1}^n\log \frac{\gamma_{\{0\}}((\omega_{-i})_0|\bla_{-\N}(\omega_{-i+1}^{-1}b_0\omega_1^{\infty})_{1}^\infty )}
    {\gamma_{\{0\}}(\bla_0|\bla_{-\N}(\omega_{-i+1}^{-1}b_0\omega_1^{\infty})_{1}^\infty )}
    \cdot
    \frac{\gamma_{\{0\}}(\bla_0|\bla_{-\N}(\omega_{-i+1}^{-1}a_0\omega_1^{\infty})_{1}^\infty )}
    {\gamma_{\{0\}}((\omega_{-i})_0|\bla_{-\N}(\omega_{-i+1}^{-1}a_0\omega_1^{\infty})_{1}^\infty )}.
\end{equation}
Then, by applying the bar moving property, one gets
\begin{multline*}
    \rho_n^{a_0,b_0}(\omega)-\rho_0^{a_0,b_0}(\omega)
    =\\
    \sum_{i=1}^n
    \log \frac{\gamma_{\{0,i\}}((\omega_{-i})_0(b_0)_i|\bla_{-\N}(\omega_{-i+1}^{-1})_1^{i-1} (\omega_1^\infty)_{i+1}^\infty ) }
    {\gamma_{\{0,i\}}(\bla_0(b_0)_i|\bla_{-\N}(\omega_{-i+1}^{-1})_1^{i-1} (\omega_1^\infty)_{i+1}^\infty ) }
    \cdot
    \frac{\gamma_{\{0,i\}}(\bla_0(a_0)_i|\bla_{-\N}(\omega_{-i+1}^{-1})_1^{i-1} (\omega_1^\infty)_{i+1}^\infty ) }
    {\gamma_{\{0,i\}}((\omega_{-i})_0(a_0)_i|\bla_{-\N}(\omega_{-i+1}^{-1})_1^{i-1} (\omega_1^\infty)_{i+1}^\infty ) }
\end{multline*}
and thus, by applying the bar moving property once more,
\begin{multline*}
\rho_n^{a_0,b_0}-\rho_0^{a_0,b_0}
    =
    \sum_{i=1}^n\log \frac{\gamma_{\{i\} }((b_0)_i | \bla_{-\N} (\omega_{-i}^{-1})_0^{i-1} (\omega_1^\infty)_{i+1}^\infty )}
    {\gamma_{\{i\} }((a_0)_i | \bla_{-\N} (\omega_{-i}^{-1})_0^{i-1} (\omega_1^\infty)_{i+1}^\infty )}
    \cdot
    \frac{\gamma_{\{i\} }((a_0)_i | \bla_{\Z_-} (\omega_{-i+1}^{-1})_1^{i-1} (\omega_1^\infty)_{i+1}^\infty )}
    {\gamma_{\{i\} }((b_0)_i | \bla_{\Z_-} (\omega_{-i+1}^{-1})_1^{i-1} (\omega_1^\infty)_{i+1}^\infty )},
\end{multline*}
here $\Z_-$ denotes $-\N\cup \{0\}$.
Hence, by the translation-invariance of $\gamma$,
\begin{eqnarray}
    \rho_n^{a_0,b_0}(\omega)-\rho_0^{a_0,b_0}(\omega)
    \notag &=&
    \sum_{i=1}^n\log\frac{\gamma_{\{0\}}(b_0|\bla_{-\infty}^{-i-1}\omega_{-i}^{-1}\omega_1^{\infty} ) }
    {\gamma_{\{0\}}(a_0|\bla_{-\infty}^{-i-1}\omega_{-i}^{-1}\omega_1^{\infty} )} 
    \cdot
    \frac{\gamma_{\{0\}}(a_0|\bla_{-\infty}^{-i}\omega_{-i+1}^{-1}\omega_1^{\infty} )}
    {\gamma_{\{0\}}(b_0|\bla_{-\infty}^{-i}\omega_{-i+1}^{-1}\omega_1^{\infty} )}\\
    \notag &=&
    \sum_{i=1}^n\log\frac{\gamma_{\{0\}}(b_0|\bla_{-\infty}^{-i-1}\omega_{-i}^{-1}\omega_1^{\infty} ) }
    {\gamma_{\{0\}}(a_0|\bla_{-\infty}^{-i-1}\omega_{-i}^{-1}\omega_1^{\infty} )} 
    -
    \sum_{i=0}^{n-1}
    \log\frac{\gamma_{\{0\}}(b_0|\bla_{-\infty}^{-i-1}\omega_{-i}^{-1}\omega_1^{\infty} ) }
    {\gamma_{\{0\}}(a_0|\bla_{-\infty}^{-i-1}\omega_{-i}^{-1}\omega_1^{\infty} )}\\
    \notag &=&
    \log\frac{\gamma_{\{0\}}(b_0|\bla_{-\infty}^{-n-1}\omega_{-n}^{-1}\omega_1^{\infty} ) }
    {\gamma_{\{0\}}(a_0|\bla_{-\infty}^{-n-1}\omega_{-n}^{-1}\omega_1^{\infty} )}
    -
    \log\frac{\gamma_{\{0\}}(b_0|\bla_{-\infty}^{-1}\omega_1^{\infty} ) }
    {\gamma_{\{0\}}(a_0|\bla_{-\infty}^{-1}\omega_1^{\infty} )}.
\end{eqnarray}
Thus
\begin{equation}
    \rho_n^{a_0,b_0}(\omega)
    =
    \log\frac{\gamma_{\{0\}}(b_0|\bla_{-\infty}^{-n-1}\omega_{-n}^{-1}\omega_1^{\infty} ) }
    {\gamma_{\{0\}}(a_0|\bla_{-\infty}^{-n-1}\omega_{-n}^{-1}\omega_1^{\infty} )}\;
    \darrow[\substack{n\to\infty\\\omega\in\Omega}]\;
    \log\frac{\gamma_{\{0\}}(b_0|\omega_{-\infty}^{-1}\omega_1^{\infty} ) }
    {\gamma_{\{0\}}(a_0|\omega_{-\infty}^{-1}\omega_1^{\infty} )}.
\end{equation}
The last limit also shows that $\gamma^{\phi_\gamma}=\gamma$.
\end{proof}

\section{Proofs}\label{proofs of main results}
The proofs of Theorem A and Theorem B will be based on two lemmas.
Our first lemma states the following.
\begin{lemma}\label{lemma: weak cohomologous property}
    Let $\phi\in C(\Omega)$ be an extensible function and $\gamma^\phi$ be the associated Gibbsian specification.
    Then the one-sided extensible function $\phi_{\gamma^\phi}\in C(\Omega_+)$ is \textit{weakly cohomologous} to $\phi$, i.e., there exists $C\in\R$ such that for all $\tau\in\M_{1,S}(\Omega)$,
    \begin{equation}
        \int_\Omega \phi_{\gamma^\phi}d\tau 
        =\int_\Omega\phi d\tau+C.
    \end{equation}
\end{lemma}
\begin{remark}
The second part of Theorem \ref{defs of associated two-sided spec and extensible potential} implies that the chain $\gamma\rightarrow \phi_\gamma\rightarrow \gamma^{\phi_\gamma}$ is closed, i.e., $\gamma=\gamma^{\phi_\gamma}$.
For an extensible function $\phi$, the diagram $\phi\rightarrow \gamma^\phi\rightarrow \phi_{\gamma^\phi}$ is, in general, not closed, i.e., it is not always true that $\phi=\phi_{\gamma^\phi}$.
\end{remark}
\begin{proof}
Note that $\phi_{\gamma^\phi}$ is a half-line function. 
For any $\omega_{\Z_+}\in \Omega_+$, one can easily check the following:
\begin{equation}\label{from extendible phi to u_gamma}
    \phi_{\gamma^\phi}(\omega_{\Z_+})
    =
    \lim_{n\to\infty}
    \sum_{i=-n}^{n}\Big[\phi\circ S^{i}(\bla_{-\infty}^{-1}\omega_{0}^\infty)
    -\phi\circ S^{i}(\bla_{-\infty}^{0}\omega_{1}^\infty)
    \Big]
\end{equation}
and the above limit is uniform on $\omega_{\Z_+}\in \Omega_+$.
For any $\tau\in \M_{1,S}(\Omega_+)$ and $n\in\N$, denote 
\begin{eqnarray*}
    I_n(\tau)
    &:=&
    \int_\Omega
    \sum_{i=-n}^{n}\Big[\phi\circ S^{i}(\bla_{-\infty}^{-1}\omega_{0}^\infty)
    -\phi\circ S^{i}(\bla_{-\infty}^{0}\omega_{1}^\infty)
    \Big]\tau(d\omega).
\end{eqnarray*}
For any $\ell_1,\ell_2\in \Z\cup\{-\infty, \infty\}$ with $\ell_1\leq \ell_2$, define a transformation 
\begin{equation}
   \Theta_{[\ell_1,\ell_2]}(\omega)_i:=\begin{cases}
       \omega_i, & \text{ if } i\notin [\ell_1,\ell_2];\\
       {\la}, & \text{ if } i\in [\ell_1,\ell_2].
   \end{cases} 
\end{equation}
If $\ell_1\nleq \ell_2 $, then $[\ell_1,\ell_2]=\emptyset$, therefore, for all $\omega\in \Omega$, we set $\Theta_{[\ell_1\ell_2]}(\omega)=\omega$.
Note that the families $\{S_{\ell}:\ell\in\Z\}$ and $\{\Theta_{[\ell_1,\ell_2]}:\ell_1,\ell_2\in\Z,\ell_1\leq \ell_2\}$ of the transformations on $\Omega$
have the following commutativity-type property:
for any $\ell_{1,2,3}\in\Z$ with $\ell_1\leq \ell_2$, 
\begin{equation}\label{imp identity on Theta}
  S_{\ell_3}\circ \Theta_{[\ell_1,\ell_2]}
   =
   \Theta_{[\ell_1-\ell_3,\ell_2-\ell_3]}\circ S_{\ell_3}.
\end{equation}
Thus $I_n(\tau)$ is written in terms of the transformations $\Theta$ as follows:
\begin{equation}
    I_n(\tau)
    =
    \sum_{i=-n}^n\int_\Omega\phi\circ \Theta_{(-\infty,-i-1]}\circ S^{i}(\omega)\tau(d\omega)
    -
    \sum_{i=-n}^n\int_\Omega\phi\circ \Theta_{(-\infty,-i]}\circ S^{i}(\omega)\tau(d\omega)
\end{equation}
Then, by the translation-invariance of measure $\tau$,
\begin{eqnarray}
    I_n(\tau)
    &=&
    \sum_{i=-n}^n\int_\Omega\phi\circ \Theta_{(-\infty,-i-1]}d\tau
    -
    \sum_{i=-n}^n\int_\Omega\phi\circ \Theta_{(-\infty,-i]}d\tau\\
    &=&
    \int_\Omega \phi\circ\Theta_{(-\infty,-n-1]} d\tau
    -
    \int_\Omega \phi\circ \Theta_{(-\infty,n]}d\tau.
\end{eqnarray}
By continuity of $\phi$, the sequences $\{\phi\circ \Theta_{(-\infty,-n-1]}(\omega)\}_{n\in\Z_+}$ and $\{\phi\circ \Theta_{(-\infty,n]}(\omega)\}_{n\in\Z_+}$ converge uniformly in $\omega\in \Omega$ to $\phi(\omega)$ and $\phi(\bla)$, respectively, as $n\to\infty$.
Thus one concludes that
\begin{equation}\label{limit of I_n(tau) integrals}
    \lim_{n\to\infty}I_n(\tau)=\int_{\Omega}\phi d\tau-\phi(\bla).
\end{equation}
Since the limit in (\ref{from extendible phi to u_gamma}) is uniform, we conclude from (\ref{limit of I_n(tau) integrals}) that
\begin{equation}
    \int_{\Omega}\phi_{\gamma^\phi}(\omega)\tau(d\omega)
    =
    \int_{\Omega}\phi(\omega)\tau(d\omega)-\phi(\bla).
\end{equation}

\end{proof}

Now, we formulate the second lemma.
\begin{lemma}\label{lemma: our VP and weak Bowen-Gibbsianity}
    Let $\gamma$ be a translation-invariant Gibbsian specification and $\phi_\gamma$ be the associated extensible function.
    Then
    \begin{itemize}
        \item[(i)] Every translation-invariant Gibbs state $\mu\in\G_S(\gamma)$ is weak Bowen-Gibbs with respect to $\phi_\gamma$, i.e., $\mu$ satisfies (\ref{eq: weak Bowen-Gibbs property}); 
        \item[(ii)] The set of translation-invariant Gibbs states for $\gamma$ coincides with the set of the equilibrium states for $\phi_\gamma$, i.e., $\G_S(\gamma)=\es(\phi_\gamma)$.
    \end{itemize}
\end{lemma}

\begin{proof}
\textbf{\textit{(i):}}
Now we shall prove that any translation-invariant DLR Gibbs measure $\mu$ prescribed by a specification $\gamma$ is weak Bowen-Gibbs relative to the potential $\phi_\gamma$.

The first part of Lemma \ref{lemma: pressure of specification and variation of spec} and translation-invariance of the specification $\gamma$ imply
\begin{equation}\label{partition functions on half line}
    \lim_{n\to\infty}-\frac{1}{n} \log \gamma_{[0,n]}(\bla_{[0,n]}|\omega_{[0,n]^c})
    =
    P^\bla(\gamma)
\end{equation}
and the convergence is uniform in $\omega\in \Omega$.

Now consider a configuration $\sigma\in \Omega_+$ and a cylindric set $[\sigma_0^n]$, then by the DLR equations, 
\begin{eqnarray}
    \mu([\sigma_0^n])
    \notag &=&
    \int_X \gamma_{[0,n]}(\sigma_{[0,n]}|\eta_{[0,n]^c})\mu(d\eta)\\
    &=&
    \int_X\frac{\gamma_{[0,n]}(\sigma_{[0,n]}|\eta_{[0,n]^c})}
    {\gamma_{[0,n]}(\bla_{[0,n]}|\eta_{[0,n]^c})}
    \cdot
    \gamma_{[0,n]}(\bla_{[0,n]}|\eta_{[0,n]^c})\mu(d\eta).
\end{eqnarray}
By translation-invariance of the specification $\gamma$, Lemma \ref{lemma: pressure of specification and variation of spec} yields that 
\begin{eqnarray}
    \mu([\sigma_0^n])
    \notag 
    &=&
    \int_X\frac{\gamma_{[0,n]}(\sigma_{[0,n]}|\eta_{[0,n]^c})}
    {\gamma_{[0,n]}(\bla_{[0,n]}|\eta_{[0,n]^c})}
    \cdot
    e^{-n P^\bla(\gamma)+o(n)}\mu(d\eta),
\end{eqnarray}
here the error factor $o(n)$ is independent on $\eta$ and only depends on $n$. 
Therefore, 
\begin{eqnarray}\label{eq: measure of cylinders approach via the ratios of spec}
    \mu([\sigma_0^n])
    &=&
    e^{-n P^\bla(\gamma)+o(n)}
    \int_X\frac{\gamma_{[0,n]}(\sigma_{[0,n]}|\eta_{[0,n]^c})}
    {\gamma_{[0,n]}(\bla_{[0,n]}|\eta_{[0,n]^c})}
    \mu(d\eta),
\end{eqnarray}
Hence, by taking into account (\ref{eq: est about decay rate of spec ratios}), we obtain that 
\begin{equation}\label{eq: meas of cylinders in terms of spec ratios with + bc}
    \mu([\sigma_0^n])
    =
    e^{-n P^\bla(\gamma)}
    \cdot
    \frac{\gamma_{[0,n]}(\sigma_{[0,n]}|\bla_{[0,n]^c})}
    {\gamma_{[0,n]}(\bla_{[0,n]}|\bla_{[0,n]^c})}
    \cdot e^{o(n)}.
\end{equation}
Using the bar moving property (\ref{eq: bar moving prop}), we have that 
\begin{eqnarray}\label{eq: spec ratios and the Birkhoff sum}
    \frac{\gamma_{[0,n]}(\sigma_{[0,n]}|\bla_{[0,n]^c})}
    {\gamma_{[0,n]}(\bla_{[0,n]}|\bla_{[0,n]^c})}
    &=&
    \prod_{i=0}^n\frac{\gamma_{[0,n]}(\bla_{[0,i)}\sigma_{[i,n]}|\bla_{[0,n]^c})}
    {\gamma_{[0,n]}(\bla_{[0,i]}\sigma_{(i,n]}|\bla_{[0,n]^c})} 
    \overset{(\ref{eq: bar moving prop})}{=}
    \prod_{i=0}^n 
    \frac{\gamma_{\{i\}}(\sigma_{i}| \sigma_{(i,n]}{\bla}_{[i,n]^c})}
    {\gamma_{\{i\}}({\bla}_{i}|\sigma_{(i,n]}{\bla}_{[i,n]^c})}\\
    \notag &=&
    \prod_{i=0}^n e^{\phi_\gamma\circ S^i(\sigma_{[0,n]}\bla_{[0,n]^c})}.
\end{eqnarray}
Note that in the last equation of (\ref{eq: spec ratios and the Birkhoff sum}), we used the fact that $\phi_\gamma$ is independent of the components in the negative half-line $-\N$.
Combining (\ref{eq: spec ratios and the Birkhoff sum}) with (\ref{eq: meas of cylinders in terms of spec ratios with + bc}), we get
\begin{equation}\label{eq: weak Bowen-Gibbs est with + b.c.}
    \mu([\sigma_0^n])
    =
    \frac{e^{S_{n+1}\phi_\gamma(\sigma_{[0,n]}\bla_{[0,n]^c})}}
    {e^{n P^\bla(\gamma)}}
    \cdot e^{o(n)}.
\end{equation}
Note that 
\begin{equation}\label{eq: var of S_n phi in terms of var of phi}
   \mathtt{ var}_n(S_{n+1}\phi_\gamma)\leq \sum_{k=0}^n \mathtt{var}_k(\phi_\gamma)
\end{equation}
and since $\phi_\gamma$ is continuous, $var_k(\phi_\gamma)\to 0$ as $k\to\infty$.
Thus $\mathtt{var}_n(S_{n+1}\phi_\gamma)=o(n)$, and hence we obtain from 
(\ref{eq: weak Bowen-Gibbs est with + b.c.}) that 
\begin{equation}\label{eq: weak Bowen-Gibbs est with any config}
    \mu([\sigma_0^n])
    =
    \frac{e^{S_{n+1}\phi_\gamma(\sigma)}}
    {e^{n P^\bla(\gamma)}}e^{o(n)}.
\end{equation}
We note that the weak Bowen-Gibbs property (\ref{eq: weak Bowen-Gibbs est with any config}) implies that $P^\bla(\gamma)=P(\phi_{\gamma})$, where $P(\phi_\gamma)$ is the topological pressure of $\phi_{\gamma}$.
\vspace{0.3cm}

\textbf{\textit{ (ii):}}
Now take any $\tau\in\M_{1,S}(\Omega)$ and $\mu\in\G_S(\gamma)$.
The relative entropy $H_n(\tau|\mu)$ is given by 
\begin{equation}\label{grouping the sum in rel ent in fin vol}
    H_n(\tau|\mu)
    =    \sum_{a_{0}^{n}\in E^{n+1}}\tau([a_{0}^{n}])\log\frac{\tau([a_{0}^{n}])}{\mu([a_{0}^{n}])}
    =
    \sum_{a_{0}^{n}\in E^{n+1}}\tau([a_{0}^{n}])\log\tau([a_{0}^{n}]).
    -
    \sum_{a_{0}^{n}\in E^{n+1}}\tau([a_{0}^{n}])\log\mu([a_{0}^{n}])\end{equation}
For the first sum in (\ref{grouping the sum in rel ent in fin vol}), one has that
\begin{equation}\label{eq: existence of specific ent}
    \frac{1}{n}\sum_{a_{0}^{n}\in E^{n+1}}\tau([a_{0}^{n}])\log\tau([a_{0}^{n}])
    \xrightarrow[n\to\infty]{} -h(\tau).
\end{equation}
For the second sum, by inserting (\ref{eq: var of S_n phi in terms of var of phi}) and using (\ref{eq: weak Bowen-Gibbs est with any config}), one has
\begin{eqnarray}
    \sum_{a_{0}^{n}\in E^{n+1}}\tau([a_{0}^{n}])\log\mu([a_{0}^{n}])
    \notag &=&
    \sum_{a_{0}^{n}\in E^{n+1}}\tau([a_{0}^{n}]) (S_{n+1}\phi_\gamma-(n+1)P^\bla(\gamma)+o(n)\\
    &=&
    (n+1)\int_\Omega\phi_\gamma d\tau -(n+1)P^\bla(\gamma) +o(n). \label{estimation for the second sum in rel ent}
\end{eqnarray}
By combining (\ref{eq: existence of specific ent}) and (\ref{estimation for the second sum in rel ent}), and since $P^\bla(\gamma)=P(\phi_{\gamma})$, one concludes that the relative entropy density rate $h(\tau|\mu)$ indeed exists and
\begin{equation}
    h(\tau|\mu)=\lim_{n\to\infty} \frac{1}{n}H_n(\tau|\mu)=-h(\tau)-\int_\Omega\phi_\gamma d\tau + P(\phi_{\gamma}).
\end{equation}
Thus the Variational Principle (Theorem \ref{VP by KLNR}) yields that $\tau$ is a Gibbs state for $\gamma$ if and only if $\tau$ is an equilibrium state for $\phi_\gamma$, i.e., $\G_S(\gamma)=\es(\phi_\gamma)$.

\end{proof}
\begin{proof}[Proof of Theorem \ref{Th A: ES for extensible potentials are DLR and Bowen-Gibbs}]
    One can easily see that weakly cohomologous potentials have the same equilibrium states because the functionals 
$\tau\in\M_{1,S}(\Omega)\mapsto h(\tau)+\int_\Omega \phi d\tau$ and $\tau\in\M_{1,S}(\Omega)\mapsto h(\tau)+\int_\Omega \phi_{\gamma^\phi} d\tau$
differ by only a constant $P(\phi_{\gamma^\phi})-P(\phi)$.
Furthermore, the weak cohomology between $\phi_{\gamma^\phi}$ and $\phi$ also yields the following
\cite{EFS1993}*{Proposition 2.34}:
\begin{equation}\label{sup norm corollary from weak coboundary}
    \lim_{n\to\infty}\frac{1}{n}\Big\rVert \sum_{i=0}^{n-1} \Big[\phi_{\gamma^\phi}-\phi -P(\phi_{\gamma^\phi})+P(\phi)\Big]\circ S^i \Big\lVert_\infty
    =
    0.
\end{equation}
Then the first part of Theorem A follows from Lemma \ref{lemma: weak cohomologous property} and the second part of Lemma \ref{lemma: our VP and weak Bowen-Gibbsianity} since the weak cohomologous potentials have the same set of equilibrium states, i.e., $\es(\phi)=\es(\phi_{\gamma^\phi})$. 
The second part of Theorem A also follows from Lemma \ref{lemma: weak cohomologous property} and Lemma \ref{lemma: our VP and weak Bowen-Gibbsianity}.
In fact, 
by applying the first and second parts of Lemma \ref{lemma: our VP and weak Bowen-Gibbsianity} to $\gamma=\gamma^\phi$, one obtains from Lemma \ref{lemma: weak cohomologous property} that for any equilibrium state $\mu\in\es(\phi)$,
\begin{equation}
    \mu([\sigma_0^{n-1}])
    =
    \frac{e^{(S_{n}\phi_{\gamma^\phi})(\sigma)}}
    {e^{n P(\phi_{\gamma^\phi})}}e^{o(n)},\;\; \sigma\in\Omega,
\end{equation}
 and by (\ref{sup norm corollary from weak coboundary}), one has that $S_{n}(\phi -P(\phi))=S_n(\phi_{\gamma^\phi}-P(\phi_{\gamma^\phi}))+o(n)$.
Hence one immediately concludes the weak Bowen-Gibbs property of $\mu$ with respect to the potential $\phi$.

\end{proof}

\begin{proof}[Proof of Theorem \ref{Th B: Gibbs states are equilibrium for extensible phi}]
    It is easy to see that the second part of Lemma \ref{lemma: our VP and weak Bowen-Gibbsianity} implies Theorem B.
\end{proof}

\section*{Acknowledgements}

The authors are grateful to Aernout van Enter for his valuable suggestions.

\end{document}